\documentclass[12pt]{article}
\usepackage[margin=1in]{geometry}
\usepackage{url}
\usepackage{graphicx}
\usepackage{subcaption}

\usepackage[T1]{fontenc}
\usepackage{fouriernc}

\usepackage{amsmath}
\usepackage{amssymb}
\usepackage{amsthm}
\usepackage{mathrsfs}

\usepackage{pgf,tikz}
\usetikzlibrary{arrows, automata, shapes, positioning}
\usetikzlibrary{decorations}
\usetikzlibrary{snakes}
\usepackage{verbatim}
\usepackage{xcolor} 

\theoremstyle{plain}
\newtheorem{theorem}{Theorem}
\newtheorem{corollary}[theorem]{Corollary}
\newtheorem{lemma}[theorem]{Lemma}
\newtheorem{proposition}[theorem]{Proposition}

\theoremstyle{definition}
\newtheorem{definition}[theorem]{Definition}

\newtheorem{problem}[theorem]{Problem}

\theoremstyle{remark}
\newtheorem{remark}[theorem]{Remark}

\DeclareMathOperator{\val}{val}
\DeclareMathOperator{\rep}{rep}
\DeclareMathOperator{\lex}{lex}

\title{The Formal Inverse of the Period-Doubling Sequence}
\author{
Narad Rampersad\thanks{The author is supported by NSERC Discovery
  Grant 418646-2012.}\\
Department of Mathematics and Statistics\\
University of Winnipeg\\
\url{n.rampersad@uwinnipeg.ca}\\
\and
Manon Stipulanti\thanks{The author is supported by FRIA Grant 1.E030.16.}\\
Department of Mathematics\\
University of Li\`ege\\
\url{m.stipulanti@uliege.be}}
\date{\today}

\begin{document}
\maketitle

\begin{abstract}

If $p$ is a prime number, consider a $p$-automatic sequence $(u_n)_{n\ge 0}$, and let $U(X) = \sum_{n\ge 0} u_n X^n \in \mathbb{F}_p[[X]]$ be its generating function. 
Assume that there exists a formal power series $V(X) = \sum_{n\ge 0} v_n X^n \in \mathbb{F}_p[[X]]$ which is the compositional inverse of $U$, i.e., $U(V(X))=X=V(U(X))$. 
The problem investigated in this paper is to study the properties of the sequence $(v_n)_{n\ge 0}$.
The work was first initiated for the Thue--Morse sequence, and more recently the case of two variations of the Baum--Sweet sequence has been treated. 
In this paper, we deal with the case of the period-doubling sequence. 
We first show that the sequence of indices at which the period-doubling sequence takes value $0$ (resp., $1$) is not $k$-regular for any $k\ge 2$.
Secondly, we give recurrence relations for its formal inverse, then we easily show that it is $2$-automatic, and we also provide an automaton that generates it. 
Thirdly, we study the sequence of indices at which this formal inverse takes value $1$, and we show that it is not $k$-regular for any $k\ge 2$ by connecting it to the characteristic sequence of Fibonacci numbers. 
We leave as an open problem the case of the sequence of indices at which this formal inverse takes value $0$.
We end the paper with a remark on the case of generalized Thue--Morse sequences. 
\end{abstract}

\section{Introduction}

Let us consider the following problem. 
Let $p$ be a prime number. 
Let $\boldsymbol{u}=(u_n)_{n\ge 0}$ be a $p$-automatic sequence and let $U(X) = \sum_{n\ge 0} u_n X^n \in \mathbb{F}_p[[X]]$ be its generating function. 
Assume that there exists a formal power series $V(X) = \sum_{n\ge 0} v_n X^n \in \mathbb{F}_p[[X]]$ which is the compositional inverse of $U$, i.e., $U(V(X))=X=V(U(X))$. 
What can be said about properties of the sequence $\boldsymbol{v}=(v_n)_{n\ge 0}$?

In~\cite{GU16}, the authors initiate the work on this problem and they consider the case where $\boldsymbol{u}=\boldsymbol{t}$ where $\boldsymbol{t}$ is the well-known Prouhet--Thue--Morse sequence. 
More precisely, they study the sequence $\boldsymbol{c}=(c_n)_{n\ge 0}$ which is the sequence of coefficients of the compositional inverse of the generating function of the sequence $\boldsymbol{t}$. 
They call this sequence $\boldsymbol{c}$ the \textit{inverse Prouhet--Thue--Morse sequence}.
The $2$-automaticity of $\boldsymbol{c}$ is easily deduced using Christol's theorem~\cite{CKMFR80}, but then they exhibit some recurrence relations satisfied by $\boldsymbol{c}$ and provide an automaton that generates $\boldsymbol{c}$.
They study two increasing sequences $\boldsymbol{a} = (a_n)_{n\ge 0}$ and $\boldsymbol{d} = (d_n)_{n\ge 0}$ respectively defined by
\[
\{a_n \mid n\in \mathbb{N}\} = \{m \in \mathbb{N} \mid c_m= 1\},
\]
and
\[
\{d_n \mid n\in \mathbb{N}\} =\{m \in \mathbb{N} \mid c_m= 0\}.
\]
In particular, they prove that $\boldsymbol{a}$ is $2$-regular, 
but that $\boldsymbol{d}$ is not $k$-regular for any $k\ge 2$.

More recently, the work has been extended to two sequences closely related to the Baum--Sweet sequence~\cite{Merta18}. 
The author obtains results similar to~\cite{GU16} for two variations of the Baum--Sweet sequence. 

In this paper, we consider the case where $\boldsymbol{u}=\boldsymbol{d}$ is the period-doubling sequence.
This sequence is defined by $d_n := \nu_2 (n + 1) \bmod{2}$, where the function $\nu_2$  is the exponent of the highest power of $2$ dividing its argument. 

\section{Background}

In this section, we recall the necessary background for this paper; see, for instance, \cite{AS03, Rigo1, Rigo2} for more details.

\subsection{Combinatorics on words}

Let $A$ be a finite \emph{alphabet}, i.e., a finite set consisting of \emph{letters}. 
A \emph{(finite) word} $w$ over $A$ is a finite sequence of letters belonging to $A$. 
If $w=w_n w_{n-1} \cdots w_0 \in A^*$ with $n\ge 0$ and $w_i \in A$ for all $i\in\{0,\ldots,n\}$, then the \emph{length $|w|$} of $w$ is $n+1$, i.e., it is the number of letters that $w$ contains.
We let $\varepsilon$ denote the empty word.  
This special word is the neutral element for concatenation of words, and its length is set to be $0$. 
The set of all finite words over $A$ is denoted by $A^*$, and we let $A^+=A^*\setminus\{\varepsilon\}$ denote the set of non-empty finite words over $A$. For any $n\ge 0$, we let $A^n$ denote the set of length-$n$ words in $A^*$. 

A finite word $w\in A^*$ is a \emph{prefix} of another finite word $z\in A^*$ if there exists $u\in A^*$ such that $z=wu$.
If $A$ is ordered by $<$, the \emph{lexicographic order} on $A^*$, which we denote by $<_{\lex}$, is a total order on $A^*$ induced by the order $<$ on the letters and defined as follows: $u <_{\lex} v$ either if $u$ is a strict prefix of $v$ or if there exist $a,b\in A$ and $p\in A^*$ such that $a<b$, $pa$ is a prefix of $u$ and $pb$ is a prefix of $v$.

If $L$ is a subset of $A^*$, then $L$ is called a \emph{language} and its \emph{complexity function} $\rho_{L} : \mathbb{N} \to \mathbb{N}$ is defined by $\rho_L(n)=L\cap A^n$.

An \emph{infinite word} $\boldsymbol{w}$ over $A$ is any infinite sequence over $A$. 
The set of all infinite words over $A$ is denoted by $A^\omega$. 
Note that in this paper infinite words are written in bold.
To avoid any confusion, the infinite word $\boldsymbol{w}=w_0 w_1 w_2 \cdots$ will be written as $\boldsymbol{w}=w_0, w_1, w_2, \ldots$ if necessary. 

If $\boldsymbol{w}\in A^\omega$, we define its \emph{sequence of run lengths} to be an infinite sequence over $\mathbb{N} \cup \{\infty\}$ giving the number of adjacent identical letters.
For example, the sequence of run lengths of $01^20^31^40^5 \cdots$  is $1,2,3,4,5,\ldots$.

A \emph{morphism} on $A$ is a map $\sigma : A^* \to A^*$ such that for all $u,v \in A^*$, we have $\sigma(uv)=\sigma(u)\sigma(v)$. 
In order to define a morphism, it suffices to provide the image of letters belonging to $A$. 
A morphism $\sigma : A^* \to A^*$  is \emph{$k$-uniform} if $|\sigma(a)|=k$ for all $a\in A$.
A $1$-uniform morphism is called a \emph{coding}.
If there is a subalphabet $C \subset A$ such that $\sigma(C) \subset C^*$, then we call the restriction $\sigma_C:= \sigma |_{ C^*} : C^* \to C^*$ of $\sigma$ to $C$ a \emph{submorphism} of $\sigma$.

A morphism $\sigma : A^* \to A^*$ is said to be \emph{prolongable} on a letter $a\in A$ if $\sigma(a)=au$ with $u\in A^+$ and $\lim\limits_{n \rightarrow +\infty} |\sigma^n(a)|=+\infty$.
If $\sigma$ is prolongable on $a$, then $\sigma^n(a)$ is a proper prefix of $\sigma^{n+1}(a)$ for all $n\ge 0$. 
Therefore, the sequence $(\sigma^n(a))_{n\ge 0}$ of finite words defines an infinite word $\boldsymbol{w}$ that is a fixed point of $\sigma$. 
In that case, the word $\boldsymbol{w}$ is called \emph{pure morphic}.
A \emph{morphic} word is the morphic image of a pure morphic word. 

Let $M$ be a matrix with coefficients in $\mathbb{N}$.
There exists permutation matrix $P$ such that $P^{-1}M P$ is a upper block-triangular matrix with square blocks $M_1, \ldots, M_s$ on the main diagonal that are either irreducible matrices or zeroes.
The \emph{Perron--Frobenius} eigenvalue of $M$ is $\max_{1 \le i \le s} \lambda_{M_i}$ where $\lambda_{M_i}$ is the Perron--Frobenius eigenvalue of the matrix $M_i$.

Let $f : A^* \to  A^*$ be a prolongable morphism having the infinite word $\boldsymbol{w}$ as a fixed point. 
Let $\alpha$ be the Perron--Frobenius eigenvalue of $M_f$. 
If all letters of $A$ occur in $\boldsymbol{w}$, then $\boldsymbol{w}$ is said to be a \emph{(pure) $\alpha$-substitutive word}. 
If $g : A^* \to B^*$ is a coding, then $g(\boldsymbol{w})$ is said to be an \emph{$\alpha$-substitutive word}. 

We say that two real numbers $\alpha, \beta >1$ are \emph{multiplicatively independent} if the only integers $k,\ell$ such that $\alpha^k = \beta^\ell$ are $k=\ell=0$. 
Otherwise, $\alpha$ and $\beta$ are \emph{multiplicatively dependent}.
The following result can be found in~\cite{Dur11}.

\begin{theorem}[Cobham--Durand]\label{thm:Cobham-Durand-2011}
Let $\alpha, \beta >1$ be two multiplicatively independent real numbers. 
Let $\boldsymbol{u}$ (resp., $\boldsymbol{v}$) be a pure $\alpha$-substitutive (resp., pure $\beta$-substitutive) word. 
Let $g$ and $g'$ be two non-erasing morphisms. 
If $\boldsymbol{w} = g(\boldsymbol{u}) = g'(\boldsymbol{v})$, then $\boldsymbol{w}$ is ultimately periodic. 
In particular, if an infinite word is $\alpha$-substitutive and $\beta$-substitutive, i.e., in the special case where $g$ and $g'$ are codings, then it is ultimately periodic.
\end{theorem}

\subsection{Abstract numeration systems, automatic sequences and regular sequences}

An \emph{abstract numeration system} (ANS) is a triple $S = (L,A,<)$ where $L$ is an infinite regular language over a totally ordered alphabet $(A,<)$.
The map $\rep_S : \mathbb{N} \to L$ is the one-to-one correspondence mapping $n \in \mathbb{N}$ onto the $(n+1)$st word in the genealogically ordered language $L$, which is called the \emph{$S$-representation} of $n$. 
The $S$-representation of $0$ is the first word in $L$. 
The inverse map is denoted by $\val_S : L \to \mathbb{N}$. 
If $w$ is a word in $L$, $\val_S(w)$ is its \emph{$S$-numerical value}.
For instance, the base-$k$ numeration system is an ANS; the Zeckendorff numeration system based on the Fibonacci numbers (with initial conditions $1$ and $2$) is also an ANS.

A \emph{deterministic finite automaton with output} (DFAO) is a
$6$-tuple $\mathcal{A}=(Q,q_0,A, \delta , B,\mu)$, where $Q$ is a finite set of \emph{states}, $q_0 \in Q$ is the \emph{initial state}, $A$ is a finite \emph{input alphabet},
$\delta : Q \times A \to Q$ is the \emph{transition function}, $B$ is a finite \emph{output alphabet}, and $\mu : Q \to B$ is the \emph{output function}.  
If $S = (L, A, <)$ is an ANS, we say that an infinite word $\boldsymbol{w} = w_0 w_1 w_2 \cdots \in B^\mathbb{N}$ is \emph{$S$-automatic} if there exists a DFAO $\mathcal{A}=(Q,q_0,A, \delta , B,\mu)$ such that $x_n = \mu(\delta(q_0, \rep_S (n)))$ for all $n \ge 0$. 
The automaton $\mathcal{A}$ is called a \emph{$S$-DFAO}.

When the $ANS$ is the base-$k$ numeration system with $k\ge 2$,
we have the following theorem of Cobham~\cite{Cobham72}.

\begin{theorem}[Cobham's theorem on automatic sequences]\label{thm:Cobham-1972}
An infinite word $\boldsymbol{w} \in B^\mathbb{N}$ is $k$-automatic if and only if there exist a $k$-uniform morphism $f : A^* \to A^*$ prolongable on a letter $a\in A$ and a coding $g : A^* \to B^*$ such that $\boldsymbol{w} = g( f^\omega (a) )$.
\end{theorem}

Let $\boldsymbol{u} = (u_n)_{n\ge 0}$ be an infinite sequence and let $k\ge 2$ be an integer. 
We define the \emph{$k$-kernel} of $\boldsymbol{u}$ to be the set of subsequences
\[
\mathcal{K}_k(\boldsymbol{u})=\{ (u_{k^i\cdot n+r})_{n\ge 0} \mid i \ge 0 \text{ and } 0\le r < k^i \}.
\]
We say that a sequence $\boldsymbol{u}$ is \emph{$k$-regular} if there exists a finite set $S$ of sequences such that every sequence in $\mathcal{K}_k(\boldsymbol{u})$ is a $\mathbb{Z}$-linear combination of sequences of $S$. 
The following properties can be found in~\cite{AS03,Schaeffer13}.

\begin{proposition}\label{pro:prop-aut-reg}
Let $k\ge 2$ be an integer.
\begin{itemize}
\item[$(1)$] If a sequence differs only in finitely many terms from a $k$-automatic sequence, then it is $k$-automatic.

\item[$(2)$] For all $m \ge 1$, a sequence is $k$-automatic if and only if it is $k^m$-automatic.

\item[$(2)$] If the integer sequence $(u_n)_{n\ge 0}$ is $k$-regular, then for all integers $m \ge 1$, the sequence $(u_n \bmod{m})_{n\ge 0}$ is $k$-automatic.

\item[$(3)$] A sequence is $k$-regular and takes on only finitely many values if and only if it is $k$-automatic.

\item[$(4)$] Let $(u_n)_{n\ge 0}$ be a $k$-regular sequence. Then for $a\ge 1$ and $b\ge 0$, the sequence $(u_{an+b})_{n\ge 0}$ is $k$-regular.

\item[$(5)$] Let $\boldsymbol{u}=(u_n)_{n\ge 0}$ be a sequence, and let $\boldsymbol{v}=(u_{n+1}-u_n)_{n\ge 0}$ be the first difference of $\boldsymbol{u}$. 
Then $\boldsymbol{u}$ is $k$-regular if and only if $\boldsymbol{v}$ is $k$-regular.
\end{itemize}
\end{proposition}

\subsection{Formal power series}

Let $k\ge 2$. 
The \emph{ring $\mathbb{F}_k[[X]]$ of formal power series with coefficients in the field $\mathbb{F}_k = \{0,1, \ldots, k-1\}$} is defined by 
\[
\mathbb{F}_k[[X]] = \left\{ \sum_{n\ge 0} a_n X^n \mid a_n \in  \mathbb{F}_k \right\}.
\]
We let $\mathbb{F}_k(X)$ denote the \emph{the field of rational functions}. 
We say that a formal series $A(X)=\sum_{n\ge 0} a_n X^n$
is \emph{algebraic (over $\mathbb{F}_k(X)$)} if there exist an integer $d \ge 1$ and polynomials $P_0(X)$, $P_1(X)$, $\ldots$, $P_d(X)$, with coefficients in $\mathbb{F}_k$ and not all zero, such that
\[
P_0+P_1 A+P_2 A^ 2+ \cdots + P_d A^d =0.
\]
With an infinite sequence $\boldsymbol{w}=(w_n)_{n \in \mathbb{N}}$ over $\{0,1,\ldots,k-1\}$, we can associate a formal series $W(X) = \sum_{n\ge 0} w_n X^n$ over $\mathbb{F}_k[[X]]$, which is called the \emph{generating function} of $\boldsymbol{w}$.
In the case where $k=p$ is a prime number, and if $w_0=0$ and $w_1$ is invertible in $\mathbb{F}_p$, then the series $W(X)$ is \emph{invertible} in $\mathbb{F}_p[[X]]$, i.e., there exists a series 
$U(X) \in \mathbb{F}_p[[X]]$ such that $W(U(X)) = X = U(W(X))$. 
The formal series $U(X)$ is called the \emph{(formal) inverse} of $W(X)$.  

\section{The period-doubling sequence}

The following definition can be found in~\cite{AS03}.

\begin{definition}
Consider the period-doubling sequence (indexed by \texttt{A096268} in \cite{Sloane})
$$
\boldsymbol{d} = (d_n)_{n\ge 0} =0 1 0 0 0 1 0 1 0 1 0 0 0 1 0 0 0 1 0 0 0 \cdots .
$$
This sequence is defined by $d_n := \nu_2 (n + 1) \bmod{2}$, where the function $\nu_2$  is the exponent of the highest power of $2$ dividing its argument. 
Alternatively, we have $\boldsymbol{d} = h^{\omega}(0)$, where $h(0) = 01$ and $h(1) = 00$. 
Since $h$ is a $2$-uniform morphism, then the period doubling sequence $\boldsymbol{d}$ is $2$-automatic. 
The $2$-DFAO drawn Figure~\ref{fig:aut-pd} generates the period-doubling sequence $\boldsymbol{d}$. 
Note that this automaton reads its input from least significant digit to most significant digit.
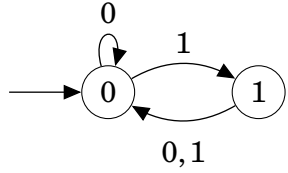
\begin{figure}
\begin{center}
\begin{tikzpicture}
\tikzstyle{every node}=[shape=circle,fill=none,draw=black,minimum size=20pt,inner sep=2pt]
\node[](a) at (0,0) {$0$};
\node[](b) at (2,0) {$1$};

\tikzstyle{every node}=[shape=circle,fill=none,draw=none,minimum size=10pt,inner sep=2pt]
\node(a0) at (-1.5,0) {};

\tikzstyle{every path}=[color =black, line width = 0.5 pt, > = triangle 45]
\tikzstyle{every node}=[shape=circle,minimum size=5pt,inner sep=2pt]
\draw [->] (a0) to [] node [] {}  (a);

\draw [->] (a) to [loop above] node [above] {$0$}  (a);
\draw [->] (a) to [bend left] node [above] {$1$}  (b);

\draw [->] (b) to [bend left] node [below] {$0,1$}  (a);

;
\end{tikzpicture}
\caption{The $2$-DFAO generating the period-doubling sequence $\boldsymbol{d}$.}
\label{fig:aut-pd}
\end{center}
\end{figure}
\end{definition}

Let us define two increasing sequences $\boldsymbol{o} = (o_n)_{n\ge 0}$ and $\boldsymbol{z} = (z_n)_{n\ge 0}$ respectively satisfying $\{o_n \mid n\in N\}=\{m\in N \mid d_m =1\}$ and $\{z_n \mid n\in N\}=\{m\in N \mid d_m =0\}$. 
We have
\begin{align*}
\boldsymbol{o}&=
1, 5, 7, 9, 13, 17, 21, 23, 25, 29, 31, 33, 37, 39, 41, 45, 49, 53, 55, 57, 61, 65, 69, 71, 73, 77, \ldots, \\
\boldsymbol{z}&= 
0, 2, 3, 4, 6, 8, 10, 11, 12, 14, 15, 16, 18, 19, 20, 22, 24, 26, 27, 28, 30, 32, 34, 35, 36, 38, 40, \ldots.
\end{align*}
Those two sequences are indexed by \texttt{A079523} and \texttt{A121539} in~\cite{Sloane}.
Observe that the binary expansions of the terms of $\boldsymbol{o}$ (resp., $\boldsymbol{z}$) end with an odd (resp., even) number of $1$'s. 
This can be seen if one considers the language accepted by the $2$-DFAO in Figure~\ref{fig:aut-pd} where the final state is the one outputting $1$ (resp., $0$).
In the following, we study the regularity of the sequences $\boldsymbol{o}$ and $\boldsymbol{z}$.

\begin{proposition}\label{pro:pd-0-not-k-reg}
The sequence $\boldsymbol{z}=(z_n)_{n\ge 0}$ is not $k$-regular for any $k \in \mathbb{N}_{\ge 2}$.
\end{proposition}
\begin{proof}
Let $\boldsymbol{\bar{d}}$ be the image of $\boldsymbol{d}$ under the exchange morphism $E:  \{0,1\}^* \to \{0,1\}^*: 0\mapsto 1, 1\mapsto 0$.
In particular, $\boldsymbol{\bar{d}}$ is the fixed point of the morphism $h'(0) = 11$ and $h'(1) = 10$ starting with $1$. 
We also have
\[
\boldsymbol{z} = \{m \in \mathbb{N} \mid d_m=0\} = \{m \in \mathbb{N} \mid \bar{d}_m=1\}.
\]

The sequence $\boldsymbol{\bar{d}}$ is related to the Thue--Morse sequence it the following way. 
Let $\boldsymbol{t}=(t_n)_{n\ge 0}$ be the Thue--Morse sequence, i.e., the fixed point of the morphism $\tau : \{0,1\}^* \to \{0,1\}^* : 0  \mapsto 01, 1 \mapsto 10$ which starts with $0$. 
In fact, the sequence $\boldsymbol{\bar{d}}$ is the first difference modulo $2$ of the Thue--Morse sequence $\boldsymbol{t}$~\cite{AS99}, i.e., $\boldsymbol{\bar{d}}=(t_{n+1}- t_n \bmod{2})_{n\ge 0}$.

In other words, the sequence $\boldsymbol{z}$ of positions of $1$'s in $\boldsymbol{\bar{d}}$ is exactly the sequence of positions in the Thue--Morse sequence $\boldsymbol{t}$ where the letters $0$ and $1$ alternate. 
Consequently, the first difference of $\boldsymbol{z}$, which is the first difference between the positions of $1$'s in $\boldsymbol{\bar{d}}$, gives the length of the blocks of consecutive identical letters in $\boldsymbol{t}$, i.e., it is the sequence of run lengths of $\boldsymbol{t}$. 

However, the sequence of run lengths of $\boldsymbol{t}$ is the sequence $\boldsymbol{p}=(p_n)_{n\ge 0}$ which is the fixed point of the morphism $f : \{1,2\}^* \to \{1,2\}^* : 1  \mapsto 121, 2 \mapsto 12221$ which starts with $1$~\cite{Aetal95}.
This sequence $\boldsymbol{p}$ is not $2$-automatic~\cite{AAS}, and by Proposition~\ref{pro:prop-aut-reg}, $\boldsymbol{p}$ is not $2^m$-automatic for any $m\ge 1$. 
Let us show that $\boldsymbol{p}$ is not $k$-automatic for any integer $k\ge 2$. 
Suppose that $\boldsymbol{p}$ is $k$-automatic for some integer $k\ge 2$ which is not a power of $2$. 
Then, by Theorem~\ref{thm:Cobham-1972}, $\boldsymbol{p}$ is the image under a coding of the fixed point of a $k$-uniform morphism whose Perron--Frobenius eigenvalue is $k$. 
Since the Perron--Frobenius eigenvalue of $f$ is $2$, then by Theorem~\ref{thm:Cobham-Durand-2011}, $\boldsymbol{p}$ is ultimately periodic, which is impossible. 

Now since $\boldsymbol{p}$ takes only two different values, $\boldsymbol{p}$ is not $k$-regular for any $k\ge 2$ by Proposition~\ref{pro:prop-aut-reg}.
Since $\boldsymbol{p}$ is the first difference of $\boldsymbol{z}$, then $\boldsymbol{z}$ is not $k$-regular for any $k\ge 2$ again by Proposition~\ref{pro:prop-aut-reg}. 
\end{proof}

The next lemma gives two other morphisms that generate the period-doubling sequence $\boldsymbol{d}$. 
Those morphisms are helpful to locate the positions of $1$'s in $\boldsymbol{d}$. 

\begin{lemma}\label{lem:pd-fixed-point}
Let $f : \{2,4\}^* \to \{2,4\}^* : 2  \mapsto 242, 4 \mapsto 24442$ and $g : \{2,4\}^* \to \{0,1\}^* : 2  \mapsto 01, 4 \mapsto 0001$.
For all $n\ge 1$, we have $h^{2n+1}(0)=g(f^n(2))$ and $h^{2n+1}(10)=g(f^n(4))$.
In particular, $\boldsymbol{d}=h^\omega(0) = g(f^\omega(2))$.
\end{lemma}
\begin{proof}
We proceed by induction on $n\ge 1$. 
The case $n=1$ can easily be checked by hand. 
Now assume that $n\ge 1$ and suppose that the result holds true for all $m\ge n$. 
We have 
\[
h^{2(n+1)+1}(0)
=h^{2n+1}(0100)
=h^{2n+1}(0)h^{2n+1}(10)h^{2n+1}(0).
\]
Now, by induction hypothesis, we find
\[
h^{2(n+1)+1}(0)
=g(f^n(2))g(f^n(4))g(f^n(2))
=g(f^n(242))
=g(f^{n+1}(2)),
\]
as expected. 
Similarly, we have
\[
h^{2(n+1)+1}(10)
=h^{2n+1}(01010100)
=h^{2n+1}(0)h^{2n+1}(10)h^{2n+1}(10)h^{2n+1}(10)h^{2n+1}(0),
\]
and by induction hypothesis, we get
\[
h^{2(n+1)+1}(0)
=g(f^n(2))g(f^n(4))g(f^n(4))g(f^n(4))g(f^n(2))
=g(f^n(24442))
=g(f^{n+1}(4)).
\]
The particular case can be deduced from the first equality of the statement. 
\end{proof}

\begin{proposition}\label{pro:pd-1-not-k-reg}
The sequence $\boldsymbol{o}=(o_n)_{n\ge 0}$ is not $k$-regular for any $k \in \mathbb{N}_{\ge 2}$.
\end{proposition}
\begin{proof}
By Lemma~\ref{lem:pd-fixed-point}, we know that $\boldsymbol{d}= g(f^\omega(2))$ with $f : \{2,4\}^* \to \{2,4\}^* : 2  \mapsto 242, 4 \mapsto 24442$ and $g : \{2,4\}^* \to \{0,1\}^* : 2  \mapsto 01, 4 \mapsto 0001$.
Observe that $|g(2)|=2$ and $|g(4)|=4$, and the letter $1$ occurs only once at the end of $g(2)$ (resp., $g(4)$). 
Consequently, the first difference of the positions of $1$'s in $\boldsymbol{d}$ -- which is the first difference of $\boldsymbol{o}$ -- is given by the shift of the sequence $f^\omega(2)$, i.e., we drop the first term. 
By the proof of Proposition~\ref{pro:pd-0-not-k-reg}, we know that $f^\omega(2)$ is not $k$-regular for any $k\ge 2$. 
By Proposition~\ref{pro:prop-aut-reg}, $\boldsymbol{o}$ is not $k$-regular for any $k\ge 2$.
\end{proof}

\begin{remark}
Using an argument similar to the one of the proof of Proposition~\ref{pro:pd-1-not-k-reg}, one can also get another way of proving Proposition~\ref{pro:pd-0-not-k-reg}. 
\end{remark}

\section{The formal inverse of the period-doubling word}

Let $D(X) = \sum_{n\ge 0} d_n X^n$ be the generating function of the period-doubling sequence $\boldsymbol{d}$. 
Since $d_0=0$ and $d_1=1$ is invertible in $\mathbb{F}_2$, then the series $D(X)$ is invertible in $\mathbb{F}_2[[X]]$, i.e., there exists a series 
$$
U(X) = \sum_{n\ge 0} u_n X^n \in \mathbb{F}_2[[X]]
$$
such that $D(U(X)) = X = U(D(X))$. 
We want to describe the sequence $\boldsymbol{u}=(u_n)_{n\ge 0}$.
Mimicking~\cite{GU16}, the first step is to get recurrence relations for the coefficients $(u_n)_{n\ge 0}$ of the series $U(X)$.
To that aim, recall the following result; see~\cite[p. 412]{CKMFR80}.

\begin{lemma}\label{lem:period-doubling-inverse}
The generating function $D(X) = \sum_{n\ge 0} d_n X^n$ of the period-doubling sequence $\boldsymbol{d}$ satisfies 
\[
 X (1 + X^2) D(X)^2 + (1 + X^2) D(X) + X =0
\]
over $\mathbb{F}_2[[X]]$.
\end{lemma}
\begin{proof}
Observe that, since $\boldsymbol{d}=h^\omega(0)$, we have $d_{2n} = 0$ and $d_{2n+1} = 1 - d_n$ for all $n\ge 0$.
Thus we have
\[
D(X) 
= \sum_{n\ge 0} d_n X^n
= \sum_{n\ge 0} d_{2n} X^{2n} + \sum_{n\ge 0} d_{2n+1} X^{2n+1} 
= X \sum_{n\ge 0} X^{2n} - X \sum_{n\ge 0} d_n X^{2n}.
\]
Now recall that, for any prime $p$ and for any series $F(X)$ in $\mathbb{F}_p[[X]]$, we have $1/(1-X) = \sum_{n\ge 0} X^n$.
Consequently,
\[
D(X) 
= \frac{X}{1 - X^2} - X D(X^2).
\]
Now working over $\mathbb{F}_2[[X]]$, we have
$$
 X (1 + X^2) D(X^2) + (1 + X^2) D(X) + X =0,
$$
and since for any prime $p$ and for any series $F(X)$ in $\mathbb{F}_p[[X]]$, we have $F(X)^p = F(X^p)$, we find
\[
 X (1 + X^2) D(X)^2 + (1 + X^2) D(X) + X =0,
 \]
 as desired.
\end{proof}

To prove the next result, we follow the method from~\cite{GU16}.

\begin{proposition}\label{pro:pol-and-rec-u}
The series $U(X)=\sum_{n\ge 0} u_n X^n$ satisfies each of the following polynomial equations 
\begin{align*}
& X^2 U(X)^3 + X U(X)^2 + (X^2 +1) U(X)  + X =0,\\
&X^3 U(X)^4 + X^3 U(X)^2 + U(X)  + X   =0
\end{align*}
over $\mathbb{F}_2[[X]]$.
In particular, the sequence $\boldsymbol{u}=(u_n)_{n\ge 0}$ verifies $u_0 = 0$, $u_1=1$, and over $\mathbb{F}_2$
\[
\begin{cases}
u_{2n} = 0 \quad \forall \, n \ge 0, \\
u_{4n+1} = u_{2n-1} \quad \forall \, n \ge 1, \\
u_{4n+3} = u_n \quad \forall \, n \ge 0. \\
\end{cases}
\]
\end{proposition}
\begin{proof}
First, let us rewrite the equation from Lemma~\ref{lem:period-doubling-inverse} in terms of $X$. 
We get 
\[
 D(X)^2 X^3 + D(X) X^2 + (D(X)^2 +1) X  + D(X) =0.
\]
In this new equation, replace $X$ by $U(X)$ to obtain
\[
 D(U(X))^2 U(X)^3 + D(U(X)) U(X)^2 + (D(U(X))^2 +1) U(X)  + D(U(X)) =0.
\]
Since $U(X)$ is the formal inverse of $D(X)$, we actually have
\begin{equation}\label{eq:period-doubling-inverse-5}
 X^2 U(X)^3 + X U(X)^2 + (X^2 +1) U(X)  + X =0,
\end{equation}
which is the first equation of the statement. 
This in turn implies that, over $\mathbb{F}_2[[X]]$,
\begin{equation}\label{eq:period-doubling-inverse-6}
U(X)^3 = \frac{X U(X)^2 + (X^2 +1) U(X)  + X}{X^2}.
\end{equation}
Now multiply~\eqref{eq:period-doubling-inverse-5} by $U(X)$ and replace $U(X)^3$ by its value~\eqref{eq:period-doubling-inverse-6}. 
We obtain first 
\[
X^2 U(X)^4 + X U(X)^3 + (X^2 +1) U(X)^2  + X U(X) =0,
\]
and so
\begin{align*}
&X^2 U(X)^4 + X \left( \frac{X U(X)^2 + (X^2 +1) U(X)  + X}{X^2} \right) + (X^2 +1) U(X)^2  + X U(X) =0 \\
&\Rightarrow X^3 U(X)^4 + X U(X)^2 + (X^2 +1) U(X)  + X + (X^3 +X) U(X)^2  + X^2 U(X) =0\\
&\Rightarrow X^3 U(X)^4 + (X^3 +2X) U(X)^2 + (2X^2 +1) U(X)  + X   =0.
\end{align*}
Working over $\mathbb{F}_2[[X]]$, this equality becomes
\[
X^3 U(X)^4 + X^3 U(X)^2 + U(X)  + X   =0 
\Leftrightarrow X^3 U(X^4) + X^3 U(X^2) + U(X)  + X   =0,
\]
which is the second equation of the statement.

Let us now prove that the recurrence relations for the sequence $\boldsymbol{u}$ hold true. 
Writing $U(X)=\sum_{n\ge 0} u_n X^{n}$ in the second equation proven above, we find
\begin{align*}
&X^3 \sum_{n\ge 0} u_n X^{4n} + X^3 \sum_{n\ge 0} u_n X^{2n} + \sum_{n\ge 0} u_n X^{n}  + X   =0 \\
\Leftrightarrow &\sum_{n\ge 0} u_n X^{4n+3} +  \sum_{n\ge 0} u_n X^{2n+3} + \sum_{n\ge 0} u_n X^{n}  + X   =0.
\end{align*}
Let us inspect the coefficients in the last equality.
We immediately have $u_0=0$ and $u_1=1$ over $\mathbb{F}_2$.
Since the exponents $4n+3$ and $2n+3$ are odd for all $n\ge 0$, we also get that, over $\mathbb{F}_2$,
$$
u_{2n} = 0 \quad \forall \, n \ge 0.
$$
Looking at the coefficient of $X^{4n+3}$, we obtain 
$$
u_n + u_{2n} + u_{4n+3} = 0 \quad \forall \, n \ge 0,
$$
which implies that $u_{4n+3} = u_n$ over $\mathbb{F}_2$ for all $n\ge 0$.
Let us now find the coefficient of $X^{4n+1}$ for $n\ge 1$. 
We have 
$$
u_{2n-1} + u_{4n+1} = 0 \quad \forall \, n \ge 1,
$$
giving $u_{4n+1} = u_{2n-1}$ over $\mathbb{F}_2$ for all $n\ge 1$.
As a consequence, the sequence $\boldsymbol{u}=(u_n)_{n\ge 0}$ verifies $u_0 = 0$, $u_1=1$, and satisfies the following recurrence relations over $\mathbb{F}_2$
\[
\left\{
\begin{array}{l}
u_{2n} = 0 \quad \forall \, n \ge 0, \\
u_{4n+1} = u_{2n-1} \quad \forall \, n \ge 1, \\
u_{4n+3} = u_n \quad \forall \, n \ge 0. \\
\end{array}
\right.
\]
\end{proof}

From now and later on, the sequence $\boldsymbol{u}=(u_n)_{n\ge 0}$ will be referred to as the \emph{inverse period-doubling sequence}, iPD sequence for short (sequence \texttt{A317542} in \cite{Sloane}).
We have 
\[
\boldsymbol{u}=(u_n)_{n\ge 0} =
0 1 0 0 0 1 0 1 0 0 0 0 0 1 0 0 0 1 0 0 0 0 0 1 0 0 0 0 0 1 0 1 0 0 0 0 0 1 0 0 0 \cdots.
\]
\begin{remark}
We have $d_n = u_n$ for all $n\le 8$, but observe that 
$$
1 = d_{4\cdot 2 + 1} = d_9 \neq u_9 = u_{4\cdot 2 + 1}u= u_{2 \cdot 2 - 1} = u_3 = 0.
$$
\end{remark}

In the following, we show that $\boldsymbol{u}$ is $2$-automatic, and we also provide an automaton that generates $\boldsymbol{u}$.

\begin{corollary}
The sequence $\boldsymbol{u}=(u_n)_{n\ge 0}$ is $2$-automatic. 
\end{corollary}
\begin{proof}
From Proposition~\ref{pro:pol-and-rec-u}, it follows that the formal power series $U(X)$ is algebraic over $\mathbb{F}_2(X)$.
By Christol's theorem, the sequence $\boldsymbol{u}$ is thus $2$-automatic.
\end{proof}

Using the following recurrence relations, the $2$-DFAO drawn in Figure~\ref{fig:aut-inverse-pd} generates the iPD sequence $\boldsymbol{u}$. 
Note that this automaton reads its input from least significant digit to most significant digit.
\begin{figure}
\begin{center}
\begin{tikzpicture}
\tikzstyle{every node}=[shape=circle,fill=none,draw=black,minimum size=20pt,inner sep=2pt]
\node[](a) at (0,0) {$0$};
\node[](b) at (2,2) {$0$};
\node[](c) at (2,-2) {$1$};
\node[](d) at (4,-1.2) {$1$};
\node[](e) at (4,1.2) {$1$};

\tikzstyle{every node}=[shape=circle,fill=none,draw=none,minimum size=10pt,inner sep=2pt]
\node(a0) at (-1.5,0) {};

\tikzstyle{every path}=[color =black, line width = 0.5 pt, > = triangle 45]
\tikzstyle{every node}=[shape=circle,minimum size=5pt,inner sep=2pt]
\draw [->] (a0) to [] node [] {}  (a);

\draw [->] (a) to [] node [above] {$0$}  (b);
\draw [->] (a) to [bend left] node [right] {$1$}  (c);

\draw [->] (b) to [loop above] node [above] {$0,1$}  (b);

\draw [->] (c) to [bend left] node [left] {$1$}  (a);
\draw [->] (c) to [] node [above] {$0$}  (d);

\draw [->] (d) to [bend left] node [left] {$0$}  (e);
\draw [->] (d) to [loop below] node [below] {$1$}  (d);

\draw [->] (e) to [bend left] node [right] {$0$}  (d);
\draw [->] (e) to [] node [above] {$1$}  (b);
;
\end{tikzpicture}
\caption{The $2$-DFAO generating the inverse period-doubling sequence $\boldsymbol{u}$.}
\label{fig:aut-inverse-pd}
\end{center}
\end{figure}
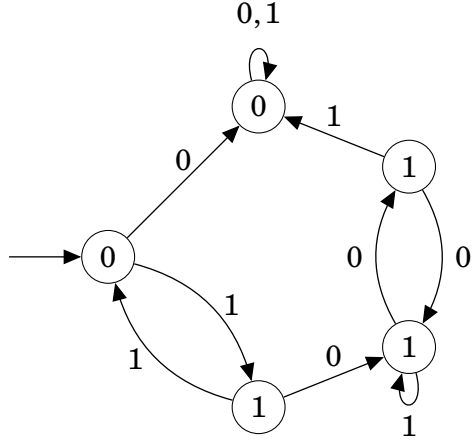

\begin{lemma}\label{lem:recurrences}
For all $n\ge 0$, $r_1\in\{0, 2\}$, $r_2\in\{0, 2,4,6\}$ and $r_3\in\{0, 2, 4, 6, 8, 10, 12, 14\}$, we have 
\begin{align}
u_n &= u_{4n+3} = u_{16n+15}, \label{eq:u-rec-1} \\
u_{2n} &= u_{4n+r_1} = u_{8n+r_2} = u_{8n+3} = u_{16n+r_3} = u_{16n+3} = u_{16n+9} = u_{16n+11} = 0,  \label{eq:u-rec-2} \\
u_{2n+1} &= u_{8n+7},  \label{eq:u-rec-3} \\
u_{4n+1} &= u_{8n+5} = u_{16n+1} = u_{16n+7} = u_{16n+13},  \label{eq:u-rec-4} \\
u_{8n+1} &= u_{16n+5}.  \label{eq:u-rec-5}
\end{align}
\end{lemma}
\begin{proof}
We make an extensive use of the recurrence relations from Proposition~\ref{pro:pol-and-rec-u}.
We show that the $2$-kernel $\mathcal{K}_2(\boldsymbol{u})$ is finitely generated by the sequences $(u_n)_{n\ge 0}$, $(u_{2n})_{n\ge 0}$, $(u_{2n+1})_{n\ge 0}$, $(u_{4n+1})_{n\ge 0}$ and $(u_{8n+1})_{n\ge 0}$.

The first equality in~\eqref{eq:u-rec-1} is directly given by Proposition~\ref{pro:pol-and-rec-u}. 
For all $n\ge 0$, we have 
$$
u_{16n+15} = u_{4(4n+3)+3} = u_{4n+3} = u_n
$$
using Proposition~\ref{pro:pol-and-rec-u} twice since $n, 4n+3\ge 0$.

Let us show~\eqref{eq:u-rec-2}.
From Proposition~\ref{pro:pol-and-rec-u}, it is clear that for all $n\ge 0$,
$$
u_{2n} = 0 = u_{4n+r_1} = u_{8n+r_2} = u_{16n+r_3}.
$$
Now for all $n\ge 0$, we have 
$$
u_{8n+3} = u_{4(2n)+3} = u_{2n} = 0,
$$ 
$$
u_{16n+3} = u_{4(4n)+3} = u_{4n} = u_{2n} = 0,
$$
and 
$$
u_{16n+11} = u_{4(4n+2)+3} = u_{4n+2} = u_{2n} = 0,
$$
using Proposition~\ref{pro:pol-and-rec-u} since $2n, 4n, 4n+2 \ge 0$.
Similarly, for all $n\ge 0$, we have $4n+2\ge 1$, thus Proposition~\ref{pro:pol-and-rec-u} gives  
$$
u_{16n+9} = u_{4(4n+2)+1} = u_{2(4n+2)-1} = u_{8n+3} = u_{2n} = 0,
$$
where the next-to-last equality comes from~\eqref{eq:u-rec-2} above. 

Let us prove~\eqref{eq:u-rec-3}.
For all $n\ge 0$, we have 
$$
u_{8n+7} = u_{4(2n+1)+3} = u_{2n+1},
$$ 
using Proposition~\ref{pro:pol-and-rec-u} since $2n+1 \ge 0$.

Let us show that~\eqref{eq:u-rec-4} holds true.
For all $n\ge 0$, we have 
$$
u_{8n+5} = u_{4(2n+1)+1} = u_{2(2n+1)-1} = u_{4n+1},
$$ 
$$
u_{16n+7} = u_{4(4n+1)+3} = u_{4n+1},
$$
and 
$$
u_{16n+13} = u_{4(4n+3)+1} = u_{2(4n+3)-1} = u_{8n+5} = u_{4n+1},
$$
using Proposition~\ref{pro:pol-and-rec-u} since $2n+1, 4n+3 \ge 1$ and $4n+1 \ge 0$.
Now we prove that $u_{16n+1}=u_{4n+1}$ for all $n\ge 0$. 
The result is trivial when $n=0$ for we have $u_{16n+1}=u_1=u_{4n+1}$.
Now suppose that $n\ge 1$. 
We first obtain from Proposition~\ref{pro:pol-and-rec-u} that
$$
u_{16n+1} = u_{4(4n)n+1} = u_{2(4n)-1} = u_{8n-1}. 
$$
Writing $n=m+1$ with $m\ge 0$, we then get
$$
u_{16n+1} = u_{8n-1} = u_{8m+7} = u_{2m+1}
$$
where the last equality comes from~\eqref{eq:u-rec-3} since $m\ge 0$.
Consequently, 
$$
u_{16n+1} = u_{2m+1} = u_{2(m+1)-1} = u_{2n-1} = u_{4n+1}
$$
using Proposition~\ref{pro:pol-and-rec-u} for the last equality since $n\ge 1$.
This gives the expected recurrence relation. 

Finally, for all $n\ge 0$, we have $4n+1 \ge 0$, so Proposition~\ref{pro:pol-and-rec-u} implies that
$$
u_{16n+5} = u_{4(4n+1)+1} = u_{2(4n+1)-1} = u_{8n+1},
$$
which proves~\eqref{eq:u-rec-5}.
\end{proof}

Since the iPD sequence $\boldsymbol{u}$ takes the values $0$ and $1$, it can also be considered as a sequence of complex numbers.
We now obtain the transcendence of its generating function. 

\begin{proposition}
The formal power series $U(X) = \sum_{n\ge 0} u_n X^n \in \mathbb{C}[[X]]$ is transcendental over $\mathbb{C}(X)$.
\end{proposition}
\begin{proof}
A classical result of Fatou states that a power series whose coefficients take only finitely many values is either rational or transcendental~\cite{Fatou}.
However, if the rational power series $A(X)=\sum_{n\ge 0} a_n X^n$ has bounded integer coefficients, then the sequence $(a_n)_{n\ge 0}$ must be ultimately periodic.
Since the iPD sequence $\boldsymbol{u}$ is not ultimately periodic, we deduce that $U(X) = \sum_{n\ge 0} u_n X^n \in \mathbb{C}[[X]]$ is transcendental over $\mathbb{C}(X)$.
\end{proof}

\section{Characteristic sequence of 1's in the iPD sequence $\boldsymbol{u}$}

In this section, we study the characteristic sequence of $1$'s in the iPD sequence $\boldsymbol{u}$.
The main result is that this sequence is not $k$-regular for any $k\ge 2$. 
Surprisingly, it is related to the characteristic sequence of Fibonacci numbers. 

\begin{definition}
Let us define an increasing sequence $\boldsymbol{a} = (a_n)_{n\ge 0}$ satisfying $\{a_n \mid n\in N\}=\{m\in N \mid u_m =1\}$ (sequence \texttt{A317543} in \cite{Sloane}). 
We have
\[
\boldsymbol{a}=1, 5, 7, 13, 17, 23, 29, 31, 37, 49, 55, 61, 65, 71, 77, 95, 101, 113, 119, 125, 127, 133, 145, \ldots. 
\]
From Proposition~\ref{pro:pol-and-rec-u}, we already know that $\boldsymbol{a}$ only contains odd integers. 
In the $2$-DFAO in Figure~\ref{fig:aut-inverse-pd}, if the states outputting $1$ are considered to be final, then the binary expansions of the terms of $\boldsymbol{a}$ is the language
\[
L_a= \{ \rep_2(a_n) \mid n\ge 0 \} = \{11\}^* 1 \cup 1 \{1,00\}^* 0 \{11\}^* 1. 
\]
For instance, $\rep_2(a_0)=1$, $\rep_2(a_1)=101$, $\rep_2(a_2)=111$, $\rep_2(a_3)=1101$.
\end{definition}

In the following, we obtain the complexity function of the language $L_a$. 
As a preliminary result, we study the language $L'=\{1,00\}^*$.

To that aim, we define the sequence $(F(n))_{n\ge 0}$ of the Fibonacci numbers with initial conditions equal to $1$ and $1$, i.e., $F(0)=1$, $F(1)=1$ and, for all $n\ge 2$, let $F(n)=F(n-1)+F(n-2)$. 
If $n\ge 1$ is an integer, a \emph{composition} of $n$ is a sequence $(a_1,a_2,\ldots,a_k)$ of positive integers, with $k\ge 1$, such that $a_1 + a_2 + \cdots + a_k = n$. The terms $a_1, a_2, \ldots,a_k$ are called \emph{the parts} of the composition. 
For example, there are eight compositions of $4$, namely $(1,1,1,1)$, $(2,1,1)$, $(1,2,1)$, $(1,1,2)$, $(3,1)$, $(1,3)$, $(2,2)$ and $(4)$. 
Observe that, among all the compositions of $4$, there are $5=F(4)$ of them whose parts are equal to $1$ or $2$.
More generally, for all $n\ge 1$, the Fibonacci number $F(n)$ counts the number of compositions of $n$ into parts equal to $1$ or $2$; see for instance~\cite[Chapter 1, Exercise 14]{Stan97}. 
Since this is equivalent to the number of strings of length $n$ in $L'$, we immediately have the following result.

\begin{lemma}\label{lem:complexity-intermediate-lang}
The complexity function $\rho_{L'} : \mathbb{N} \to \mathbb{N}$ of the language $L'$ satisfies $\rho_{L'}(n)=F(n)$ for all $n\ge 0$.
\end{lemma}

In the next result (easily proven by induction), we establish two useful equalities.

\begin{lemma}\label{lem:nbr-Fib}
For all $n\ge 1$, $\sum_{\ell=0}^{n-1} F(2\ell) = F(2n-1)$ and, for all $n\ge 2$, $\sum_{\ell=0}^{n-2} F(2\ell+1) = F(2(n-1))-1$.
\end{lemma}

\begin{proposition}\label{pro:complexity-L-a}
The complexity function $\rho_{L_a} : \mathbb{N} \to \mathbb{N}$ of the language $L_a$ satisfies $\rho_{L_a}(0)=0=\rho_{L_a}(2)$, $\rho_{L_a}(1)=1$, $\rho_{L_a}(2n)=F(2n-2)-1$ for all $n\ge 2$, and $\rho_{L_a}(2n+1)=F(2n-1)+1$ for all $n\ge 1$.
\end{proposition}
\begin{proof}
Let us define $L_{a,1} = \{11\}^* 1$ and $L_{a,2} = 1 \{1,00\}^* 0 \{11\}^* 1$. 
Since these two languages are disjoint, we have 
\[
\rho_{L_a}(n) = \rho_{L_{a,1}}(n) + \rho_{L_{a,2}}(n) \quad \forall \; n \ge 0. 
\]
In the remainder of the proof, we study the functions $\rho_{L_{a,1}}$ and $\rho_{L_{a,2}}$ separately. 
First, it is clear that
\[
\rho_{L_{a,1}}(n) =
\left\{
\begin{array}{cc}
1, & \text{if } n \text{ is odd};  \\
0, & \text{otherwise}. \\
\end{array}
\right.
\]
Now observe that $\rho_{L_{a,2}}(n)=0$ for $n\in\{0,1,2\}$.
Any word $w$ in $L_{a,2}$ is of length at least $3$ and can be factorized as $w=1u0v1$ where $u \in \{1,00\}^*$ and $v \in \{11\}^* $. 
In the following, this highlighted $0$ between $u$ and $v$ will play an important role. 
Since $v$ is of even length, then the position of $0$ in $w=1u0v1$ is odd (we start indexing words at $0$). 

Let $n\ge 1$. 
Now take $w=w_{2n} w_{2n-1} \cdots w_0 \in L_{a,2}$ with $w_i \in \{0,1\}$ and $|w|=2n+1$. 
Then we have $w_{2n}=1=w_0$ and there exists an odd integer $0 < i < 2n$ such that $w_i =0$ and 
\[
w = 1 w_{2n-1} w_{2n-2} \cdots w_{i+1} 0  w_{i-1}  w_{i-2} \cdots w_1 1. 
\]
with $u=w_{2n-1} w_{2n-2} \cdots w_{i+1} \in \{1,00\}^*$ and $v=w_{i-1}  w_{i-2} \cdots w_1 \in \{11\}^*$.
Consequently, for a fixed $i$, the number of different words of length $2n+1$ of the previous form in $L_{a,2}$ is given by the number of different words of length $|u|=2n-1-i$ in $L'$. 
We thus obtain 
\begin{align*}
\rho_{L_{a,2}}(2n+1) 
&=\sum_{\substack{0 < i < 2n \\ i \text{ odd}}} \rho_{L'}(2n-1-i) \\
&= \sum_{j=0}^{n-1} \rho_{L'}(2n-1-(2j+1))
= \sum_{j=0}^{n-1} \rho_{L'}(2(n-1-j)) \\
&= \sum_{\ell=0}^{n-1} \rho_{L'}(2\ell)
= \sum_{\ell=0}^{n-1} F(2\ell) \\
&= F(2n-1)
\end{align*}
where the last two equalities come from Lemmas~\ref{lem:complexity-intermediate-lang} and~\ref{lem:nbr-Fib}. 

Let $n\ge 2$. 
Now take $w=w_{2n-1} w_{2n-2} \cdots w_0 \in L_{a,2}$ with $w_i \in \{0,1\}$ and $|w|=2n$. 
The reasoning in this case is similar to the previous one. 
Then we have $w_{2n-1}=1=w_0$ and there exists an odd integer $0 < i < 2n-1$ such that $w_i =0$ and 
\[
w = 1 w_{2n-2} w_{2n-3} \cdots w_{i+1} 0  w_{i-1}  w_{i-2} \cdots w_1 1. 
\]
with $u=w_{2n-2} w_{2n-3} \cdots w_{i+1} \in \{1,00\}^*$ and $v=w_{i-1}  w_{i-2} \cdots w_1 \in \{11\}^*$.
Consequently, for a fixed $i$, the number of different words of length $2n$ of the previous form in $L_{a,2}$ is given by the number of different words of length $|u|=2n-2-i$ in $L'$. 
We thus obtain 
\begin{align*}
\rho_{L_{a,2}}(2n) 
&=\sum_{\substack{0 < i < 2n-1 \\ i \text{ odd}}} \rho_{L'}(2n-2-i) \\
&= \sum_{j=0}^{n-2} \rho_{L'}(2n-2-(2j+1))
= \sum_{j=0}^{n-2} \rho_{L'}(2(n-2-j)+1) \\
&= \sum_{\ell=0}^{n-2} \rho_{L'}(2\ell+1)
= \sum_{\ell=0}^{n-2} F(2\ell+1) \\
&= F(2n-2)-1
\end{align*}
where the last two equalities come from Lemmas~\ref{lem:complexity-intermediate-lang} and~\ref{lem:nbr-Fib}.

Finally, we find 
\begin{align*}
&\rho_{L_a}(0) = \rho_{L_{a,1}}(0) + \rho_{L_{a,2}}(0) = 0  + 0 = 0,\\
&\rho_{L_a}(1) = \rho_{L_{a,1}}(1) + \rho_{L_{a,2}}(1) = 1 + 0 = 1, \\
&\rho_{L_a}(2) = \rho_{L_{a,1}}(2) + \rho_{L_{a,2}}(2) = 0 + 0 = 0, \\
&\rho_{L_a}(2n+1) = \rho_{L_{a,1}}(2n+1) + \rho_{L_{a,2}}(2n+1) = 1 + F(2n-1) \quad \forall \, n \ge 1, \\
&\rho_{L_a}(2n) = \rho_{L_{a,1}}(2n) + \rho_{L_{a,2}}(2n) = 0+ F(2n-2)-1 = F(2n-2)-1 \quad \forall \, n \ge 2.
\qedhere
\end{align*}
\end{proof}

The sequence $(a_n \bmod{3})_{n\ge 0}$ shows a particularly unexpected behavior as explained in the next two results. 

\begin{lemma}\label{lem:a-n-mod-3}
Let $n\ge 0$. Then $a_n \bmod{3} \equiv r$ with $r\in \{1,2\}$.
More precisely, let $w_n := \rep_2(a_n)$.
If $w_n \in L_{a,1}$, or if $w_n \in L_{a,2}$ and $|w_n|$ is even, then $a_n \bmod{3} \equiv 1$; if $w_n \in L_{a,2}$ and $|w_n|$ is odd, then $a_n \bmod{3} \equiv 2$.  
\end{lemma}
\begin{proof}
First, we have 
\begin{align}\label{eq:PowersOf2Mod3}
(2^n \bmod{3})_{n \ge 0} = (1, -1, 1, -1, 1, -1,\ldots). 
\end{align}
Now let $n\ge 0$ and set $w_n := \rep_2(a_n)$.
If $w_n \in L_{a,1}$, then from~\eqref{eq:PowersOf2Mod3} we deduce that $a_n \bmod{3} \equiv 1$.
Assume that $w_n \in L_{a,2}$ and write $w_n=p_ns_n$ with $p_n \in 1\{1,00\}^*$ and $s_n \in 0\{11\}^*1$.
Since $|s_n|$ is even, then~\eqref{eq:PowersOf2Mod3} shows that $\val_2(s_n) \bmod{3} \equiv 1$.

As first case, suppose that $|w_n|$ is odd. 
Then $|p_n|$ is also odd, and so $p_n$ contains an odd number of $1$'s separated by even-length blocks of $0$'s.  
Because the $0$'s blocks have even length, the contributions of successive $1$'s in $p_n$ alternate in value between $+1 \bmod{3}$ and $-1 \bmod{3}$.  
Since $|s_n|$ is even, after reading $s_n$ then reading $p_n$ gives an additional $+1 \bmod{3}$.
Consequently, both $p_n$ and $s_n$ together give $2 \bmod{3}$, i.e.,  $a_n \bmod{3} \equiv \val_2(p_ns_n) \bmod{3} \equiv 2$. 

As a second case, assume that $|w_n|$ is even. 
Then $|p_n|$ is even, and so $p_n$ contains an even number of $1$'s separated by even-length blocks of $0$'s.  
Again the $1$'s in $p_n$ contribute alternating $+1 \bmod{3}$ and $-1 \bmod{3}$, and since there is an even number of them, the $1$'s in $p_n$ contribute $0 \bmod{3}$ in total.  
Thus, in this case, $a_n \bmod{3} \equiv \val_2(p_ns_n) \bmod{3} \equiv 1$.
\end{proof}

\begin{proposition}\label{pro:a-n-mod-3-Fib}
The sequence $(a_n \bmod{3})_{n\ge 0}$ is given by the infinite word
\[
1^{F(0)} 2^{F(1)} 1^{F(2)} 2^{F(3)} 1^{F(4)} 2^{F(5)} \cdots.
\]
In particular, the sequence of run lengths of $(a_n \bmod{3})_{n\ge 0}$ is the sequence of Fibonacci numbers $(F(n))_{n\ge 0}$ .
\end{proposition}
\begin{proof}
Recall that $L_a^n = L_a \cap \{0,1\}^n$ denotes the set of length-$n$ words in $L_a$. 
We can order the words of $L_a^n$ by lexicographic order, i.e.,
\[
L_a^n = \{ w_{n,1} <_{\lex} w_{n,2} <_{\lex} \cdots <_{\lex} w_{n,\#L_a^n} \}.
\] 
By Proposition~\ref{pro:complexity-L-a}, $\# L_a^0=0=\# L_a^2$, $\# L_a^1=1=F(0)$, 
$\# L_a^{2n} = F(2n-2)-1$ for all $n\ge 2$, and $\# L_a^{2n+1} = F(2n-1)+1$ for all $n\ge 1$.

Let us first consider $L_a^{2n}$ for $n\ge 2$. 
From Lemma~\ref{lem:a-n-mod-3}, 
we know that $\val_2(w_{2n,i}) \bmod{3} \equiv 1$ for all $i\in\{1,2,\ldots, F(2n-2)-1\}$.
In other terms, we get 
\[
(\val_2(w_{2n,i}) \bmod{3})_{1 \le i \le F(2n-2)-1} = 1^{F(2n-2)-1}.
\]

Let us now study $L_a^{2n+1}$ for $n\ge 0$.
In the case where $n=0$, then $L_a^{1}=\{w_{1,1}\}$ with $w_{1,1}=1$, which of course gives $\val_2(w_{1,1}) \bmod{3}=1^{F(0)}$.
Assume that $n\ge 1$. 
Since the words of $L_a^{2n+1}$ are ordered lexicographically, we know that $w_{2n+1,i} \in L_{a,2}$ for all $i\in\{1,2,\ldots, F(2n-1)\}$, and $w_{2n+1,F(2n-1)+1}=1^{2n+1} \in L_{a,1}$.
From Lemma~\ref{lem:a-n-mod-3}, 
we obtain that $\val_2(w_{2n+1,i}) \bmod{3} \equiv 2$ for all $i\in\{1,2,\ldots, F(2n-1)\}$, and $\val_2(w_{2n+1,F(2n-1)+1}) \bmod{3} \equiv 1$.
In fact, we obtain 
\[
(\val_2(w_{2n+1,i}) \bmod{3})_{1 \le i \le F(2n-1)+1} = 2^{F(2n-1)}1.
\]

Observe that, for any $n\ge 1$, concatening the sequences $(\val_2(w_{2n+1,i}) \bmod{3})_{1 \le i \le F(2n-1)+1}$ and $(\val_2(w_{2n+2,i}) \bmod{3})_{1 \le i \le F(2n)-1}$ gives
$(2^{F(2n-1)}1) \cdot (1^{F(2n)-1}) = 2^{F(2n-1)} 1^{F(2n)}.$
Now putting everything together, we find
\begin{align*}
(a_n \bmod{3})_{n\ge 0}
&= 1^{F(0)} \cdot 2^{F(1)}1 \cdot 1^{F(2)-1} \cdot 2^{F(3)}1 \cdot 1^{F(4)-1} 2^{F(5)}1 \cdots \\
&= 1^{F(0)}  2^{F(1)} 1^{F(2)}  2^{F(3)}  1^{F(4)} 2^{F(5)} \cdots,
\end{align*}
as expected.
\end{proof}

To show that $\boldsymbol{a}$ is not $k$-regular for any $k\ge 2$, the idea is to study the sequence of consecutive differences in $(a_n \bmod{3})_{n\ge 0}$.
Let us define the sequence $\boldsymbol{\delta}=(\delta_n)_{n\ge 0}$ by 
\[
\delta_n =
\left\{
\begin{array}{ll}
1, & \text{if } (a_{n+1} - a_n) \bmod{3}\neq 0; \\
0, & \text{otherwise}.
\end{array}
\right.
\] 
From Proposition~\ref{pro:a-n-mod-3-Fib}, we know that $\delta_n = 1$ if and only if there exists $n=F(m)-2$ for some $m\ge 0$.
If we let $\boldsymbol{x}$ denote the characteristic sequence of Fibonacci numbers, i.e., $x_n$ equals $1$ if $n$ is a Fibonacci number, $0$ otherwise, then $\boldsymbol{\delta}=(x_n)_{n\ge 2}$ since for all $n\ge 0$
\[
\delta_n = 1
\Leftrightarrow n= F(m) - 2 \text{ for some } m\ge 0
\Leftrightarrow n +2 = F(m) \text{ for some } m\ge 0
\Leftrightarrow x_{n+2} =1.
\]
The goal is now to show that $\boldsymbol{x}$ is not $k$-automatic for any $k\ge 2$; then the non-$k$-automaticity of $\boldsymbol{\delta}$ can easily be deduced.
What follows is widely inspired by~\cite{Rigo1, Rigo2}. 
In our context, we consider the ANS $(L_F,\{0,1\},<)$ where $L_F=\{\varepsilon\} \cup 1\{0,01\}^*$ is the language of Fibonacci representations of nonnegative integers with $0 < 1$. 
Observe that the DFA $\mathcal{A}$ in Figure~\ref{fig:FibonacciLanguage} accepts the regular language $L_F$.
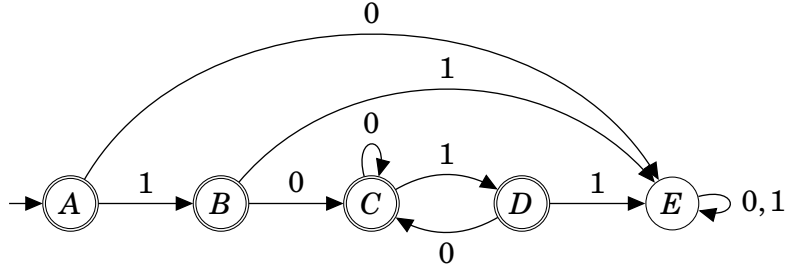
\begin{figure}
\begin{center}
\begin{tikzpicture}
\tikzstyle{every node}=[shape=circle,fill=none,draw=black,minimum size=20pt,inner sep=2pt]
\node[accepting](a) at (0,0) {$A$};
\node[accepting](b) at (2,0) {$B$};
\node[accepting](c) at (4,0) {$C$};
\node[accepting](d) at (6,0) {$D$};
\node[](e) at (8,0) {$E$};

\tikzstyle{every node}=[shape=circle,fill=none,draw=none,minimum size=10pt,inner sep=2pt]
\node(a0) at (-1,0) {};

\tikzstyle{every path}=[color =black, line width = 0.5 pt, > = triangle 45]
\tikzstyle{every node}=[shape=circle,minimum size=5pt,inner sep=2pt]
\draw [->] (a0) to [] node [] {}  (a);

\draw [->] (a) to [bend left=60] node [above] {$0$}  (e);
\draw [->] (a) to [] node [above] {$1$}  (b);

\draw [->] (b) to [bend left=50] node [above] {$1$}  (e);
\draw [->] (b) to [] node [above] {$0$}  (c);

\draw [->] (c) to [loop above] node [] {$0$}  (c);
\draw [->] (c) to [bend left] node [above] {$1$}  (d);

\draw [->] (d) to [bend left] node [below] {$0$}  (c);
\draw [->] (d) to [] node [above] {$1$}  (e);

\draw [->] (e) to [loop right] node [] {$0,1$}  (c);
;
\end{tikzpicture}
\caption{The DFA $\mathcal{A}$ accepting the language $\{\varepsilon\} \cup 1\{0,01\}^*$.}
\label{fig:FibonacciLanguage}
\end{center}
\end{figure}

\begin{lemma}\label{lem:Fib-automatic}
The characteristic sequence of Fibonacci numbers $\boldsymbol{x}$ is Fibonacci-automatic.
\end{lemma}
\begin{proof}
The Fibonacci-DFAO $\mathcal{B}$ in Figure~\ref{fig:FibonacciNumbers} generates the sequence $\boldsymbol{x}$ in the Zeckendorff numeration system. 
In particular, this shows that $\boldsymbol{x}$ is Fibonacci-automatic. 
\begin{figure}
\begin{center}
\begin{tikzpicture}
\tikzstyle{every node}=[shape=circle,fill=none,draw=black,minimum size=20pt,inner sep=2pt]
\node[](a) at (0,0) {$0_0$};
\node[](b) at (2,0) {$1$};
\node[](c) at (4,0) {$0_1$};

\tikzstyle{every node}=[shape=circle,fill=none,draw=none,minimum size=10pt,inner sep=2pt]
\node(a0) at (-1,0) {};

\tikzstyle{every path}=[color =black, line width = 0.5 pt, > = triangle 45]
\tikzstyle{every node}=[shape=circle,minimum size=5pt,inner sep=2pt]
\draw [->] (a0) to [] node [] {}  (a);

\draw [->] (a) to [loop above] node [above] {$0$}  (a);
\draw [->] (a) to [] node [above] {$1$}  (b);

\draw [->] (b) to [loop above] node [above] {$0$}  (b);
\draw [->] (b) to [] node [above] {$1$}  (c);

\draw [->] (c) to [loop above] node [] {$0,1$}  (c);

;
\end{tikzpicture}
\caption{The Fibonacci-DFAO $\mathcal{B}$ generating $\boldsymbol{x}$.}
\label{fig:FibonacciNumbers}
\end{center}
\end{figure}
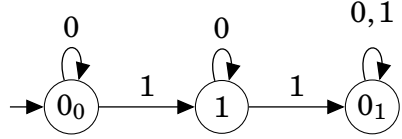
\end{proof}

When a word is $S$-automatic for some ANS $S$, then it is in fact morphic~\cite{Rigo2}.

\begin{theorem}\label{thm:morphic-automatic-ANS}
An infinite word $\boldsymbol{w}$ is morphic if and only if $\boldsymbol{w}$ is $S$-automatic for some ANS $S$.
\end{theorem}  

From Lemma~\ref{lem:Fib-automatic} and Theorem~\ref{thm:morphic-automatic-ANS}, we easily deduce that $\boldsymbol{x}$ is morphic. 
More precisely, we want to build the morphisms that generate $\boldsymbol{x}$.
We follow the constructive proof of Theorem~\ref{thm:morphic-automatic-ANS} (we refer the reader to~\cite[Chapter 2]{Rigo2} for more details).  
 
\begin{lemma}\label{lem:Fib-morphic-1}
Let $f : \{z,a_0,a_1,\ldots,a_7\}^* \to \{z,a_0,a_1,\ldots,a_7\}^*$ be the morphism defined by $f(z)=za_0$ and
\[
\begin{array}{c|cccccccc}
i & 0 & 1 & 2 & 3 & 4 & 5 & 6 & 7 \\
\hline
f(a_i) & a_1a_2 & a_1 a_4 & a_3 a_7 & a_3a_6 & a_4 a_7 & a_5 a_6 & a_5 a_7 & a_7 a_7
\end{array}.
\]
We also define the morphism $g : \{z,a_0,a_1,\ldots,a_7\}^* \to \{0,1\}^*$ by $g(z)=g(a_1)=g(a_4)=g(a_7)=\varepsilon$, $g(a_0)=g(a_5)=g(a_6)=0$ and $g(a_2)=g(a_3)=1$.
Then $\boldsymbol{x} = g( f^\omega(z) )$.
In particular, the word $\boldsymbol{x}$ is morphic.
\end{lemma}
\begin{proof}
First recall that the DFA $\mathcal{A}$ in Figure~\ref{fig:FibonacciLanguage} accepts the language $L_F=\{\varepsilon\} \cup 1\{0,01\}^*$, and the Fibonacci-DFAO $\mathcal{B}$ in Figure~\ref{fig:FibonacciNumbers} generates the sequence $\boldsymbol{x}$.
Then, the product automaton $\mathcal{P}=\mathcal{A}\times \mathcal{B}$ is drawn in Figure~\ref{fig:Product}. 
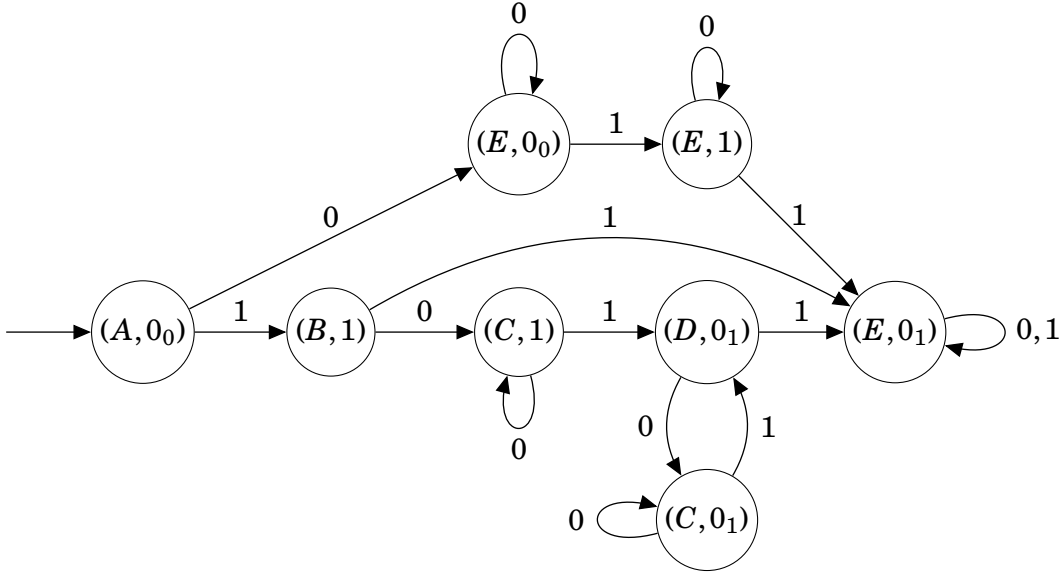
\begin{figure}
\begin{center}
\begin{tikzpicture}
\tikzstyle{every node}=[shape=circle,fill=none,draw=black,minimum size=20pt,inner sep=2pt]
\node[](a) at (0,0) {$(A,0_0)$};
\node[](b) at (2.5,0) {$(B,1)$};
\node[](c) at (5,0) {$(C,1)$};
\node[](d) at (7.5,0) {$(D,0_1)$};
\node[](e) at (7.5,-2.5) {$(C,0_1)$};
\node[](f) at (5,2.5) {$(E,0_0)$};
\node[](g) at (7.5,2.5) {$(E,1)$};
\node[](h) at (10,0) {$(E,0_1)$};

\tikzstyle{every node}=[shape=circle,fill=none,draw=none,minimum size=10pt,inner sep=2pt]
\node(a0) at (-2,0) {};

\tikzstyle{every path}=[color =black, line width = 0.5 pt, > = triangle 45]
\tikzstyle{every node}=[shape=circle,minimum size=5pt,inner sep=2pt]
\draw [->] (a0) to [] node [] {}  (a);

\draw [->] (a) to [] node [above] {$0$}  (f);
\draw [->] (a) to [] node [above] {$1$}  (b);

\draw [->] (b) to [] node [above] {$0$}  (c);
\draw [->] (b) to [bend left] node [above] {$1$}  (h);

\draw [->] (c) to [loop below] node [below] {$0$}  (c);
\draw [->] (c) to [] node [above] {$1$}  (d);

\draw [->] (d) to [] node [above] {$1$}  (h);
\draw [->] (d) to [bend right] node [left] {$0$}  (e);

\draw [->] (e) to [loop left] node [left] {$0$}  (e);
\draw [->] (e) to [bend right] node [right] {$1$}  (d);

\draw [->] (f) to [loop above] node [above] {$0$}  (f);
\draw [->] (f) to [] node [above] {$1$}  (g);

\draw [->] (g) to [loop above] node [above] {$0$}  (g);
\draw [->] (g) to [] node [above] {$1$}  (h);

\draw [->] (h) to [loop right] node [right] {$0,1$}  (h);

;
\end{tikzpicture}
\caption{The DFA $\mathcal{P}$ which is the product of $\mathcal{A}$ and $\mathcal{B}$.}
\label{fig:Product}
\end{center}
\end{figure}
If we set
\begin{align*}
&a_0:=(A,0_0), a_1:=(E,0_0), a_2:=(B,1), a_3:=(C,1), \\
&a_4:=(E,1), a_5:=(C,0_1), a_6:=(D,0_1), a_7:=(E,0_1),
\end{align*}
then we can associate a morphism $\psi_\mathcal{P} : \{z,a_0,a_1,\ldots,a_7\}^* \to \{z,a_0,a_1,\ldots,a_7\}^*$ with $\mathcal{P}$ as follows.
It is defined by $\psi_\mathcal{P}(z)=za_0$ and
\[
\begin{array}{c|cccccccc}
i & 0 & 1 & 2 & 3 & 4 & 5 & 6 & 7 \\
\hline
\psi_\mathcal{P}(a_i) = \delta_{\mathcal{P}}(a_i,0)\delta_{\mathcal{P}}(a_i,1) & a_1a_2 & a_1 a_4 & a_3 a_7 & a_3a_6 & a_4 a_7 & a_5 a_6 & a_5 a_7 & a_7 a_7
\end{array}
\]
where $\delta_\mathcal{P}$ is the transition function of $\mathcal{P}$. 
Notice that $\psi_\mathcal{P} = f$. 
We also define the morphism
\[
g : \{z,a_0,a_1,\ldots,a_7\}^* \to \{0,1\}^* :
z,a_1,a_4,a_7 \mapsto\varepsilon; a_0,a_5,a_6 \mapsto 0; a_2,a_3 \mapsto 1.
\]
It is well known that $\boldsymbol{x} = g( f^\omega(z) )$, which shows that $\boldsymbol{x}$ is morphic.
\end{proof}

Observe that the morphism $g$ in Lemma~\ref{lem:Fib-morphic-1} is erasing, i.e., the image of some letter is the empty word. 
In the following lemma (see~\cite[Chapter 3]{Rigo1}), we get rid of the erasure and we later obtain two new non-erasing morphisms that generate $\boldsymbol{x}$.

\begin{lemma}\label{lem:getting-rid-of-erasing-morphism}
Let $\boldsymbol{w} = g(f^\omega(a))$ be a morphic word where $g : B^* \to A^*$ is a (possibly erasing) morphism and $f : B^* \to B^*$ is a non-erasing morphism. 
Let $C$ be a subalphabet of $\{b \in B \mid g(b)=\varepsilon \}$ such that $f_C$ is a submorphism of $f$.
Let $\lambda_C : B^* \to B^*$ be the morphism defined by $\lambda_C (b)=\varepsilon$ if $b\in C$, and $\lambda_C (b) = b$ otherwise. 
The morphisms $f_\varepsilon := (\lambda_C \circ  f) |_{ (B \setminus C)^*}$ and $g_\varepsilon := g |_{  (B \setminus C)^*}$ are such that $\boldsymbol{w} = g_\varepsilon(f_\varepsilon^\omega(a))$.
\end{lemma}

\begin{proposition}\label{pro:Fib-morphic-2}
Let  $\phi:\{a,b,c,d,e\}^*\to \{a,b,c,d,e\}^*$ be the morphism defined by
\[
\phi:\{a,b,c,d,e\}^*\to \{a,b,c,d,e\}^*: 
\left\{ 
\begin{array}{ccc}
a & \mapsto & ab, \\
b & \mapsto & c, \\
c & \mapsto & ce, \\
d & \mapsto & de, \\
e & \mapsto & d
\end{array}
\right.
\]
and let $\mu:\{a,b,c,d,e\}^*\to \{0,1\}^*: a,d,e \mapsto 0; b,c \mapsto 1$ be a coding.
Then $\boldsymbol{x} = \mu ( \phi^\omega( a ) )$.
\end{proposition}
\begin{proof}
We make use of Lemmas~\ref{lem:Fib-morphic-1} and~\ref{lem:getting-rid-of-erasing-morphism}. 
First, we have 
\[
\{ b \in \{z,a_0,a_1,\ldots,a_7\} \mid g(b)=\varepsilon\} = \{z,a_1,a_4,a_7\},
\]
so we choose $C=\{a_1,a_4,a_7\}$ for $f_C$ is a submorphism of $f$.
Then the morphism 
\[
f_\varepsilon : \{z,a_0,a_2,a_3,a_5,a_6\}^* \to \{z,a_0,a_2,a_3,a_5,a_6\}^* 
\]
is defined by $f_\varepsilon(z)=za_0$, $f_\varepsilon(a_0)=a_2$, $f_\varepsilon(a_2)=a_3$, $f_\varepsilon(a_3)=a_3 a_6$, $f_\varepsilon(a_5)=a_5 a_6$ and $f_\varepsilon(a_6)=a_5$, while the morphism $g_\varepsilon : \{z,a_0,a_2,a_3,a_5,a_6\}^* \to \{0,1\}^*$ is given by $g_\varepsilon(z)=\varepsilon$, $g_\varepsilon(a_0)=g_\varepsilon(a_5)=g_\varepsilon(a_6)=0$ and $g_\varepsilon(a_2)=g_\varepsilon(a_3)=1$.
We also have $\boldsymbol{x} = g_\varepsilon(f_\varepsilon^\omega(z))$.
Note that $f_\varepsilon|_{\{a_2,a_3,a_5,a_6\}^*}$ is a submorphism of $f_\varepsilon$.

Let us define the morphism $f'_\varepsilon : \{a_0,a_2,a_3,a_5,a_6\}^* \to \{a_0,a_2,a_3,a_5,a_6\}^* $
by $f'_\varepsilon(a_0)=a_0a_2$, and $f'_\varepsilon = f_\varepsilon|_{\{a_2,a_3,a_5,a_6\}^*}$.
From that definition, $f'_\varepsilon$ is prolongable on $a_0$. 
Also consider the morphism $g'_\varepsilon : \{a_0,a_2,a_3,a_5,a_6\}^* \to \{0,1\}^*$ given by $g'_\varepsilon=g_\varepsilon|_{\{a_0,a_2,a_3,a_5,a_6\}^*}$.
We have
\begin{align*}
f_\varepsilon^\omega(z) 
&= z a_0 f_\varepsilon(a_0) f_\varepsilon^2(a_0) f_\varepsilon^3(a_0) f_\varepsilon^4(a_0) \cdots \\
&= z a_0 f_\varepsilon(a_0) f_\varepsilon(f_\varepsilon(a_0)) f_\varepsilon^2(f_\varepsilon(a_0)) f_\varepsilon^3(f_\varepsilon(a_0))  \cdots \\
&= z a_0 a_2 f_\varepsilon(a_2) f_\varepsilon^2(a_2) f_\varepsilon^3(a_2)  \cdots \\
&= z a_0 a_2 f'_\varepsilon(a_2) (f'_\varepsilon(a_2))^2 (f'_\varepsilon(a_2))^3 \cdots,
\end{align*}
thus we get
\begin{align*}
\boldsymbol{x}
&= g_\varepsilon( f_\varepsilon^\omega(z) ) \\
&= g_\varepsilon(z) g_\varepsilon(a_0) g_\varepsilon(a_2) g_\varepsilon(f'_\varepsilon(a_2)) g_\varepsilon((f'_\varepsilon(a_2))^2) g_\varepsilon( (f'_\varepsilon(a_2))^3 ) \cdots \\
&= \varepsilon g'_\varepsilon(a_0) g'_\varepsilon(a_2) g'_\varepsilon(f'_\varepsilon(a_2)) g'_\varepsilon((f'_\varepsilon(a_2))^2) g'_\varepsilon( (f'_\varepsilon(a_2))^3 ) \cdots  \\
&= g'_\varepsilon ( a_0 a_2 f'_\varepsilon(a_2) (f'_\varepsilon(a_2))^2 (f'_\varepsilon(a_2))^3 \cdots ) \\
&= g'_\varepsilon ( (f'_\varepsilon)^\omega(a_0) ).
\end{align*}
Up to a renaming of the letters, we have proven the claim. 
\end{proof}

\begin{corollary}
Let $\varphi = \frac{1}{2} ( \sqrt{5} +1 )$ be the golden ratio.
The word $\boldsymbol{x}$ is $\varphi$-substitutive. 
\end{corollary}
\begin{proof}
Let 
\[
M_\phi
=
\left(
\begin{array}{ccccc}
1 & 1 & 0 & 0 & 0 \\
0 & 0 & 1 & 0 & 0 \\
0 & 0 & 1 & 0 & 1 \\
0 & 0 & 0 & 1 & 1 \\
0 & 0 & 0 & 1 & 0 \\
\end{array}
\right)
\]
be the matrix associated with the morphism $\phi$. 
The Perron--Frobenius eigenvalue of $M_\phi$ is $\varphi = \frac{1}{2} ( \sqrt{5} +1 )$. 
Since all the letters of $\{a,b,c,d,e\}$ occur in $\phi^{\omega}(a)$, then $\boldsymbol{x}$ is $\varphi$-substitutive by Proposition~\ref{pro:Fib-morphic-2}.
\end{proof}

\begin{proposition}\label{prop:not-k-automatic}
The sequence $\boldsymbol{x}$ is not $k$-automatic for any $k \in \mathbb{N}_{\ge 2}$.
\end{proposition}
\begin{proof}
Proceed by contradiction and suppose that there exists an integer $k \ge 2$ such that $\boldsymbol{x}$ is $k$-automatic.
Then, by Theorem~\ref{thm:Cobham-1972}, $\boldsymbol{x}$ is also $k$-substitutive. 
Indeed, it is not difficult to see that the Perron--Frobenius eigenvalue of the matrix associated with a $k$-uniform morphism is the integer $k$.
Clearly, $k$ and $\varphi$ are two multiplicatively independent real numbers. 
Thus, by Theorem~\ref{thm:Cobham-Durand-2011}, $\boldsymbol{x}$ is ultimately periodic. 
This is impossible.  
\end{proof}

\begin{corollary}
The sequence $(a_n)_{n\ge 0}$ is not $k$-regular for any $k \in \mathbb{N}_{\ge 2}$.
\end{corollary}
\begin{proof}
Suppose that the sequence $(a_n)_{n\ge 0}$ is $k$-regular for some $k \ge 2$. 
Then by Proposition~\ref{pro:prop-aut-reg}, the sequence $(a_n \bmod{3})_{n\ge 0}$ is $k$-automatic, and so is $\boldsymbol{x}$. 
This contradicts Proposition~\ref{prop:not-k-automatic}.
\end{proof}

We end this section with the following open problem. 

\begin{problem}
Let us define an increasing sequence $\boldsymbol{b} = (b_n)_{n\ge 0}$ satisfying $\{b_n \mid n\in N\}=\{m\in N \mid u_m =0\}$ (sequence \texttt{A317544} in \cite{Sloane}). 
We have
\[
\boldsymbol{b}=
0, 2, 3, 4, 6, 8, 9, 10, 11, 12, 14, 15, 16, 18, 19, 20, 21, 22, 24, 25, 26, 27, 28, 30, 32, 33, 34, 35, \ldots. 
\]
Is the sequence $\boldsymbol{b}$ $k$-regular for some $k\ge 2$?
\end{problem}

\section{A remark on the case of generalized Thue--Morse sequences}

Let $p$ be a prime number and define $s_p : \mathbb{N} \to \mathbb{N}$ to be the sum-of-digits function in base $p$. 
Define the sequence $(t_p(n))_{n\ge 0}$ by $t_p(n) = s_p(n) \bmod{p}$.
When $p=2$, then $(t_2(n))_{n\ge 0}$ is the Thue--Morse sequence. 
For that reason, the sequences $(t_p(n))_{n\ge 0}$ are called \emph{generalized Thue--Morse sequences}~\cite{AS03}.
For a fixed $p$, also define the generating function $T_p(X) = \sum_{n\ge 0} t_p(n)X^n$ of $(t_p(n))_{n\ge 0}$.
Observe that, for all primes $p$, we have $t_p(0) = s_p(0) \bmod{p} = 0$ and $t_p(1) = s_p(1) \bmod{p} = 1$. 
Since $1$ is invertible in $\mathbb{F}_p$, the series $T_p(X)$ is invertible in $\mathbb{F}_p[[X]]$, i.e., there exists a series 
$$
U_p(X) = \sum_{n\ge 0} u_{p,n} X^n \in \mathbb{F}_p[[X]]
$$
such that $T_p(U_p(X)) = X = U_p(T_p(X))$. 
Now, from~\cite[Example 12.1.3]{AS03}, we know that 
\begin{equation}\label{eq:gen-TM-inverse-1}
(1 - X)^{p+1} T_p(X)^p - (1 - X)^2 T_p(X) + X = 0.
\end{equation}
Studying $T_p(X)$ and $U_p(X)$ is part of~\cite[Problem 5.5]{GU16}.

As a first attempt, one could try to use the method from~\cite{GU16}, mimicking the case of the classical Thue--Morse sequence. 
In~\eqref{eq:gen-TM-inverse-1}, the leading exponent of $X$ is $p+1$ since $\binom{p+1}{p+1}=1$ in $\mathbb{F}_p$. 
Thus the first step of the method presented in~\cite{GU16} gives an equation with a leading term (in terms of $X$) equal to $T_p(X)^p X^{p+1}$.
When replacing $X$ by $U_p(X)$, we get a new equation with a leading term (in terms of $U_p(X)$ this time) equal to $X^p U_p(X)^{p+1}$.
Multiplying this by $U_p(X)$ gives a term involving $U_p^{p+2}$, which cannot be compared to $U_p(X^{p+2})$ in $\mathbb{F}_p[[X]]$ for a general $p$. 

The goal is to transform the polynomial equation that we initially
obtain for $U_p(X)$ into one where the powers of $U_p(X)$
all have exponents that are powers of $p$ (as we did, for example, in
the second equation of Proposition~\ref{pro:pol-and-rec-u}).  In
fact, such a polynomial equation always exists: this claim is known as
Ore's Lemma (see \cite[Lemma~12.2.3]{AS03}) and is an important step
in the proof of Christol's Theorem.  Adamczewski and Bell~\cite[Lemmas
8.1, 8.2]{AdamBell12} give an effective procedure for obtaining a
polynomial equation of this form, which provides one possible
strategy for analyzing the series $U_p(X)$; however, the method
described by Adamczewski and Bell could result in a polynomial
equation for $U_p(X)$ whose coefficients (which are elements of
$\mathbb{F}_p[X]$) might potentially have quite large degrees.

\end{document}